\documentclass[11pt]{amsart}
\usepackage[dvipdfmx]{graphicx}
\usepackage{comment}
\usepackage{amssymb}
\usepackage{mathrsfs}
\usepackage[matrix,arrow]{xy}
\usepackage{amsmath}
\usepackage{amssymb}
\usepackage{amsfonts}
\usepackage{amsthm}
\usepackage{setspace}
\usepackage{amscd}
\usepackage{tikz}
\usetikzlibrary{cd}
\usepackage{lscape}
\usepackage{graphics}
\usepackage{indentfirst}
\usepackage{xcolor}
\usepackage{autobreak}
\usepackage{chngcntr}
\usepackage{comment}
\usepackage{enumerate}
\usepackage{subfigure}
\usepackage[dvipdfmx]{hyperref}

\newtheorem{thm}{Theorem}[section]
\newtheorem{prop}[thm]{Proposition}
\newtheorem{lem}[thm]{Lemma}
\newtheorem{cor}[thm]{Corollary}

\newtheorem{quest}[thm]{Question}

\newtheorem{com}[thm]{Comment}

\theoremstyle{definition}

\newtheorem{rem}[thm]{Remark}

\newcommand{\bN}{\mathbb{N}}

\newcommand{\bR}{\mathbb{R}}
\newcommand{\bZ}{\mathbb{Z}}

\newcommand{\hana}{\mathfrak}

\newcommand{\hen}{\mathcal}
\newcommand{\nsl}{\trianglelefteq}

\newcommand{\varSi}{\varSigma}

\usepackage{graphicx}
\usepackage{geometry}
\title[{}]
{{Smooth finite group actions on homology six-spheres 
with odd Euler characteristic fixed point sets}}

\author[S.~Tamura]{Shunsuke Tamura}

\address{{National Institute of Technology},
{Tsuyama College},
{624-1 Numa, Tsuyama},
{Okayama 708-8509 Japan}}

\email{tamura\_s@tsuyama.kosen\textrm{-}ac.jp}

\counterwithin{table}{section}

\begin{document}

\maketitle

\begin{abstract}
In this paper, we prove that 
if a finite group $G$ acts smoothly and effectively 
on an integral homology $6$-sphere 
and the $G$-fixed-point set has an odd Euler characteristic,
then the acting group $G$ is isomorphic to either
the alternating group $A_5$ on $5$ letters,
the symmetric group $S_5$ on $5$ letters,
or the Cartesian product $A_5 \times C_2$,
where $C_2$ is a group of order $2$,
and the $G$-fixed-point set consists 
of precisely one point.
\end{abstract}

\section{Introduction}

Throughout this paper,
let $G$ be a finite group and $S^n$ the standard $n$-sphere,
and both of manifolds and group actions on manifolds
are assumed to be smooth.
Let $\bZ$ be the ring of integers and 
$\bZ_{p}$ the finite field of order $p$,
where $p$ is a prime.
For $R = \bZ$ or $\bZ_{p}$,
an {\it $R$-homology $n$-sphere} refers to
an $n$-dimensional closed manifold
that has the same $R$-homology groups as the $n$-sphere $S^n$.
For simplicity,
a $\bZ$-homology sphere will be called 
a {\it homology sphere}.
For a nonnegative integer $m$ and a manifold $M$ with $G$-action,
we call the $G$-action on $M$
an {\it $m$-fixed-point $G$-action}
if the $G$-fixed-point set $M^{G}$ consists of exactly $m$-points (including the empty set).
When $m=1$ (resp. $m=0$, $m \equiv 1\;{\rm mod}\;2$),
an $m$-fixed-point $G$-action is called 
a {\it one-fixed-point $G$-action} 
(resp. a {\it fixed-point-free $G$-action}, an {\it odd-fixed-point $G$-action}).
In this paper, as a generalization of 
one-fixed-point and odd-fixed-point $G$-actions,
we study  {\it odd-Euler-characteristic $G$-actions},
which are $G$-actions whose $G$-fixed-point sets 
have odd Euler characteristics.

Various topologists have long studied the following question:

\begin{quest}\label{quest:1}
For a given pair $(G,n)$ of
a finite group $G$ and a nonnegative integer $n$,
does there exist a one-fixed-point 
or an odd-fixed-point $G$-action 
on the $n$-sphere $S^n$?
\end{quest}

\noindent
The motivation for this study was a remark
by D.~Montgomery--H.~Samelson \cite[7.~Remarks]{MonSam} that
there would not exist a one-fixed-point $G$-action on a sphere.
This remark led to the study of the existence of 
one-fixed-point $G$-actions on spheres.
The first breakthrough was due to 
E.~Floyd--R.~Richardson \cite{FloRic},
who obtained fixed-point-free $A_5$-actions 
on high-dimensional disks.
Every fixed-point-free $G$-action on a disk yields 
a topological one-fixed-point $G$-action on a sphere 
by identifying the boundary of the disk,
and the existence of a smooth one was thus expected.
Next, R.~Oliver \cite{Oli} completely determined the finite groups $G$ 
that have a fixed-point-free $G$ on some disk.
Every one-fixed-point $G$-action on a sphere 
gives a fixed-point-free $G$-action on a disk 
by removing an open $G$-slice neighborhood 
of the unique $G$-fixed point.
Thus, Oliver's result provided a necessary condition for 
finite groups to have a one-fixed-point $G$-action on a sphere.
This condition was later proved to be  
a necessary and sufficient condition, see \cite{LaiMor}.
We give the definition of the finite groups 
determined by Oliver, referred to as {\it Oliver groups}.
Let $p$ and $q$ be primes or $1$ (possibly $p=q$).
Let $\hen{G}_{p}^{q}$ denote the family 
of finite groups $G$ having a normal sequence 
$P \nsl H \nsl G$ such that 
$P$ is a $p$-group, $H/P$ is cyclic,
and $G/H$ is a $q$-group,
where the case $p=1$ (resp. $q=1$) 
indicates $|P|=1$ (resp. $|G/H|=1$),
and let $\hen{G}$ denote the union of 
the families $\hen{G}_{p}^{q}$,
where $p$ and $q$ run over all primes.
The Oliver groups are defined as 
the finite groups not belonging to $\hen{G}$.
For example, nonsolvable finite groups are Oliver groups.
The first example of a one-fixed-point $G$-action on a sphere
was discovered by E.~Stein \cite{Ste}, namely,
he constructed one-fixed-point actions on $S^7$
of the binary icosahedral group $SL(2,5)$.
Subsequently,
T.~Petrie \cite{Pet1, Pet2} proved that 
for each positive integer $m$,
there are $m$-fixed-point actions 
on certain high-dimensional homotopy spheres 
of abelian Oliver groups of odd order
and specific nonsolvable finite groups.

For a nonnegative integer $k$,
a $G$-action on a manifold $M$ is called {\it $k$-pseudofree}
if for each nontrivial subgroup $H$ of $G$,
the dimension of each connected component of $M^{H}$ 
is less than or equal to $k$.
In particular, a $k$-pseudofree $G$-action on $M$ 
is said to be {\it properly}
if there exists a subgroup $H$ of $G$ such that
some connected component of $M^{H}$ has dimension $k$.
E.~Laitinen--P.~Traczyk \cite{LaiTra} showed that 
if there exists a $2$-pseudofree one-fixed-point $G$-action on 
a homotopy sphere $S$ of dimension $\geq 5$,
then the $G$-action on $S$ must be 
a one-fixed-point $A_5$-action on $S^6$. 
Soon after, M.~Morimoto \cite{Mor3,Mor4} proved that
there are actually one-fixed-point $A_5$-actions 
on $S^6$ and $S^n$, $n \geq 9$,
and A.~Bak--M.~Morimoto \cite{BakMor1, BakMor2} later obtained
one-fixed-point $A_5$-actions on $S^7$ and $S^8$.
On the other hand,
according to the results of M.~Furuta \cite{Fur},
S.~Demichelis \cite{Dem},
and N.~P.~Buchdahl--S.~Kwasik--R.~Schultz \cite{BucKwaSch},
there are no one-fixed-point actions of finite groups 
on homotopy spheres of dimensions $\leq 5$.
Thus, the $6$-sphere $S^6$ is the least dimensional sphere 
admitting a one-fixed-point $G$-action.
However, 
it had been not known whether 
finite groups $G$ other than $A_5$ 
have an effective one-fixed-point $G$-action on $S^6$.
A.~Borowicka--P.~Mizerka \cite{BorMiz} provided
some examples of Oliver groups $G$ 
not having effective one-fixed-point $G$-actions 
on homology $6$-spheres.
Recently, M.~Morimoto \cite{MorNew2} proved that 
there exists a $3$-pseudofree one-fixed-point $G$-action on 
a homotopy sphere $S$ of dimension $k = 2l \geq 6$
if and only if 
$G$ is isomorphic to $A_5$, $S_5$, or $A_5 \times C_2$ and
$S$ is diffeomorphic to $S^6$.
Based on this result, he proposed a conjecture that
a finite group $G$ which can act effectively on $S^6$ with exactly one fixed point 
would be isomorphic to either $A_5$, $S_5$, or $A_5 \times C_2$.

The following two theorems 
give an affirmative answer to his conjecture
and also answer Question \ref{quest:1} in the case $n=6$.

\begin{thm}\label{thm:main}
If a homology $6$-sphere $\varSigma$ has an orientation-preserving 
effective odd-Euler-characteristic $G$-action,
then $G$ is isomorphic to $A_5$ and $\varSigma^{G}$ consists of exactly one point.
\end{thm}

Let $\varSigma$ be a homology $6$-sphere and 
suppose that $\varSigma$ has
an orientation-reversing effective odd-Euler-characteristic $G$-action.
Let $H$ be the subgroup of $G$ consisting of the elements $g \in G$ 
such that the induced diffeomorphisms of $\varSigma$ are orientation-preserving.
As $[G:H] = 2$,
it holds that 
$\chi(\varSigma^{G}) \equiv \chi(\varSigma^{H}) 
\equiv 1\;{\rm mod}\;2$ (by Lemma \ref{lem:modp}).
By Theorem \ref{thm:main},
we have $H \cong A_5$ and $\varSigma^{H}=\{x\}$,
and hence $G$ is isomorphic to $S_5$ or $A_5 \times C_2$ 
and $\varSigma^{G} = \varSigma^{H} = \{x\}$.
Thus we obtain the following:

\begin{thm}\label{thm:main2}
If a homology $6$-sphere $\varSigma$ has 
an orientation-reversing effective 
odd-Euler-characteristic $G$-action,
then $G$ is isomorphic to $S_5$ or $A_5 \times C_2$
and $\varSigma^{G}$ consists of exactly one point.
\end{thm}

\begin{rem}\label{rem:main}
In the situation of Theorem \ref{thm:main} (resp. \ref{thm:main2}),
one can also prove that 
an effective one-fixed-point action on $\varSigma$ 
of $A_5$ (resp. $S_5$ or $A_5 \times C_2$) must be 
properly $2$-pseudofree (resp. properly $3$-pseudofree).
This will be proved in Lemma \ref{lem:A5} (resp. Proposition \ref{prop:main}).
\end{rem}

\begin{com}
The author does not know whether there exists a $G$-action on a homology $6$-sphere $\varSigma$ 
such that $\chi(\varSigma^{G})$ is neither $0$, $1$, nor $2$. 
It would be interesting if such an action does exist.
\end{com}

We present some results on odd-fixed-point $G$-actions on homology spheres
that served in our proof of Lemma \ref{lem:nonabelian} stated later.
Inspired by a result of A.~Borowiecka \cite{Bor},
M.~Morimoto--the author \cite{MorTam}
introduced odd-fixed-point actions,
as a generalization of one-fixed-point actions,
and provided dimensions of homology spheres 
not admitting effective odd-fixed-point $G$-actions
for $G = S_5$ or $SL(2,5)$.
For this result for $G=S_5$,
M.~Morimoto \cite{MorNew3} constructed
effective one-fixed-point $G$-actions on $S^n$
for the dimensions $n$ other than those obtained in \cite{MorTam}.
The author \cite{Tam}, P.~Mizerka \cite{Miz},
and M.~Morimoto \cite{MorNew1}
also gave dimensions of homology spheres 
not admitting effective one-fixed-point 
or odd-fixed-point $G$-actions
for several nonsolvable finite groups.
However,
it is not yet known whether the dimensions of homology spheres
given in \cite{MorTam} for $G= SL(2,5)$,
\cite{Tam}, \cite{Miz}, and \cite{MorNew1} are the best ones.

Let a finite group $G$ act on a manifold $M$ with a $G$-fixed point $x$ of $M$.
The tangent space $T_{x}(M)$ inherits the $G$-action on $M$ linearly
(in other words, $T_{x}(M)$ can have a structure of an $\bR[G]$-module),
and $T_{x}(M)$ is called 
the {\it tangential $G$-module} at $x \in M^{G}$.
Every tangent space $T_{x}(M)$ appearing in this paper 
should be understood as the tangential $G$-module 
at a $G$-fixed point $x$ of $M$.
Our work is principally group-theoretic and 
the method of our proof of Theorem \ref{thm:main} is based on 
the proof of \cite[Theorem {\rm I\hskip-0.15cm}I.4]{BucKwaSch}.
Their proof proceeded by studying the tangential $G$-module 
at a $G$-fixed point
and minimal normal subgroups of the acting group $G$ on a manifold.
In our proof of Theorem \ref{thm:main},
we also study the unique maximal nilpotent normal subgroup $F(G)$
of the acting group $G$ on a homology $6$-sphere,
where $F(G)$ is often called the {\it Fitting subgroup} of $G$.
It is known that every minimal normal subgroup of a finite group is 
a characteristically simple group, i.e.
the Cartesian product of isomorphic finite simple groups 
(see \cite[Theorems 1.4 and 1.5, p.16--17]{Gor}).
Thus, each minimal normal subgroup of $G$ is either 
a nonabelian characteristically simple group or 
contained in the Fitting subgroup $F(G)$.
Therefore, Theorem \ref{thm:main}
will be obtained from the following two lemmas:

\begin{lem}\label{lem:nonabelian}
Suppose that $G$ has a nonabelian minimal normal subgroup.
If a homology $6$-sphere $\varSigma$ has 
an orientation-preserving effective 
odd-Euler-characteristic $G$-action,
then $G$ is isomorphic to $A_5$ and $\varSigma^{G}$ consists of exactly one point.
\end{lem}

\begin{lem}\label{lem:abelian}
If a homology $6$-sphere $\varSigma$ has 
an orientation-preserving effective 
odd-Euler-characteristic $G$-action and
the Fitting subgroup $F(G)$ is nontrivial,
then $\chi(\varSi^{G}) \equiv 0\;{\rm mod}\;2$. 
\end{lem}

We remark that an assertion similar to Lemma \ref{lem:abelian} 
not necessarily hold 
for homology spheres of dimensions $7$ or higher,
because there exist orientation-preserving 
effective one-fixed-point actions of $SL(2,5) \times C_{m}$ on $S^n$, 
where $n=3+4k$ with $k \in \bN$ and $C_{m}$ is a cyclic group of order $m$ with $(m,30)=1$,
see \cite[Theorem 1.3]{MorNew2}.
 
Lemmas \ref{lem:nonabelian} and \ref{lem:abelian} 
will be proved throughout Sections \ref{nonabelian} and \ref{abelian}, respectively.


\section{Preliminaries}\label{basic}

In this section we provide some results which will be used 
in Sections \ref{nonabelian} and \ref{abelian}.

The following three lemmas are well-known results in our study.

\begin{lem}[{\rm \cite[Chapter I\hspace{-.1em}I\hspace{-.1em}I, Theorem\;4.3]{Bre})}]
\label{lem:modp}
Let $p$ be a prime, $P$ a $p$-group, and $M$ a compact manifold with $P$-action.
Then
\[
\chi(M) \equiv \chi(M^{P}) \;{\rm mod}\;p.
\]
\end{lem}

\begin{lem}[{\rm Smith's theorem \cite{Smi},
\cite[Chapter I\hspace{-.1em}I\hspace{-.1em}I, Theorem 5.1]{Bre}}]\label{lem:smith}
Let $p$ be a prime and $P$ a $p$-group.
If $P$ acts on a $\bZ_{p}$-homology sphere $\Xi$,
then $\Xi^{P}$ is itself a $\bZ_{p}$-homology sphere.
\end{lem}

\begin{lem}[{\rm cf. \cite[Proposition 2.1]{Oli}}]\label{lem:euler}
Let $p$ and $q$ be primes and 
$\Xi$ a $\bZ_p$-homology sphere with $G$-action.
If $G$ belongs to $\hen{G}^{q}_{p}$ 
then $\chi(\Xi^{G}) \equiv \chi(S^{k})\;{\rm mod}\;q$ for some $k$.
In particular,
if $G$ belongs to $\hen{G}^{2}_{p}$ 
then $\chi(\Xi^{G}) \equiv 0\;{\rm mod}\;2$.
\end{lem}

\begin{proof}
See the proof of \cite[Proposition 2.2]{MorNew1}.
\end{proof}

The following two lemmas are simple but useful in our study.

\begin{lem}\label{lem:lowsphere}
Let $p$ be a prime and $\Xi$ a $\bZ_p$-homology sphere.
If $G$ acts on $\Xi$ with a $G$-fixed point $x$ of $\Xi$ and
there is a normal subgroup $H$ of $G$ 
such that $\Xi^{H}$ is a sphere of dimension $\leq 2$ 
or that $H$ is a $p$-group and $\dim T_{x}(\Xi)^{H} \leq 2$,
then $\Xi^{G}$ is a sphere of dimension $\leq 2$.
\end{lem}

\begin{proof}
The assertion follows since $\Xi^{G} = (\Xi^{H})^{G/H}$ and 
the fact that nonempty fixed point sets 
of spheres of dimensions $\leq 2$ 
are spheres of dimensions $\leq 2$.
\end{proof}

\begin{lem}\label{lem:onedim}
Let $G$ act on a closed manifold $M$ 
with a $G$-fixed point.
If there exists a normal subgroup $H$ of $G$ such that
$\dim T_{x}(M)^{H}=1$ for each $x \in M^{G}$,
then $\chi(M^{G})$ is a nonnegative even integer.
\end{lem}

\begin{proof}
For $x \in M^{G}$, let $M^{H}_{x}$ denote 
the connected component of $M^{H}$ including $x$.
Since each connected component $M^{H}_{x}$ 
is a $1$-dimensional closed connected manifold,
i.e. $M^{H}_{x} \cong S^1$, with $G/H$-action,
the Euler characteristic of
\[
M^{G} = \coprod_{x \in M^{G}} (M^{H}_{x})^{G/H}
\]
is a nonnegative even integer.
\end{proof}

Let $\rho : G \to SO(n)$ be a faithful real representation of degree $n$
and let $U$ be the $n$-dimensional faithful $\bR[G]$-module
with a $G$-invariant inner product
determined by $\rho$, up to isomorphism.
Let $L$ be a normal subgroup of $G$.
Since the $L$-fixed-point set $U^{L}$ is an $\bR[G/L]$-module,
$U$ decomposes into $U^{L}$
 and its orthogonal complement $U_{L}$ as $\bR[G]$-modules,
i.e. $U=U^{L} \oplus U_{L}$.
Thus, $\rho : G \to SO(n)$ also decomposes into two subrepresentations
\[
\rho^{L} : G \to O(l) \; {\rm and}\; \rho_{L} : G \to O(m)
\]
corresponding to $U^{L}$ and $U_{L}$, respectively. 
Here, $l=\dim U^{L}$ and $m=\dim U_{L}$.
Let $H$ and $K$ denote the kernels of $\rho^{L}$ and $\rho_{L}$, respectively.
Since $H$ (resp. $K$) is the kernel of $\rho^{L}$ (resp. $\rho_{L}$),
the quotient representation
\[
\rho^{L}/H : G/H \to O(l)\; ({\rm resp.}\,\, \rho_{L}/K : G/K \to O(m))
\]
is faithful.
As $\rho$ is faithful, it follows that $H \cap K = E$.
Therefore,
the quotient projection from $G$ to $G/H \times G/K$ is injective
and we have the following commutative diagram
\[
\xymatrix{
G \ar[r] \ar[rrrd]_{\rho}  & G/H \times G/K  \ar[rr]^{\rho^{L}/H \times \rho_{L}/K} 
& {}\ar@{}[dr]|\circlearrowleft   & O(l) \times O(m)  \ar[d]^{\subset} \\
{}        & {} & {} & SO(n). \\
}
\]
The restriction $\rho|_{H}$ (resp. $\rho|_{K}$) 
is orientation-preserving and faithful
and the restriction ${\rho^{L}}|_{H}$ (resp. $\rho_{L}|_{K}$) is trivial.
This implies that the restriction ${\rho_{L}}|_{H}$ (resp. $\rho^{L}|_{K}$) 
is orientation-preserving and faithful.
Therefore, $H$ (resp. $K$) is isomorphic to a finite subgroup of $SO(m)$ (resp. $SO(l)$).
Putting these together, we obtain the following proposition.

\begin{prop}\label{prop:SOV}
Let $L$ be a normal subgroup of $G$ and $\rho : G \to SO(n)$ a faithful real representation of degree $n$.
Let $H$ and $K$ be the kernels of $\rho^{L}$ and $\rho_{L}$, respectively.
Then the following holds:
\begin{enumerate}
  \item[$(1)$] The intersection of $H$ and $K$ is trivial, i.e. $H \cap K = E$.
  \item[$(2)$] The projection $G \to G/H \times G/K$ is injective.
  \item[$(3)$] $H$ is isomorphic to a finite subgroup of $SO(m)$ and it contains $L$.
  \item[$(4)$] $K$ is isomorphic to a finite subgroup of $SO(l)$.
  \item[$(5)$] $G/H$ is isomorphic to a finite subgroup of $O(l)$.
  \item[$(6)$] $G/K$ is isomorphic to a finite subgroup of $O(m)$.
\end{enumerate}
Here, $l$ and $m$ are degrees of $\rho^{L}$ and $\rho_{L}$, respectively.
\end{prop}

The method of Proposition \ref{prop:SOV} appeared 
in the proof of \cite[Theorem {\rm I\hspace{-.01em}I}.4., Assertion 3]{BucKwaSch}.
Applying Proposition \ref{prop:SOV} for $L =F(G)$,
we have the following:

\begin{prop}\label{prop:FG}
In the situation of Proposition \ref{prop:SOV},
if $L=F(G)$ then the Fitting subgroup $F(K)$ of $K$ is trivial.
\end{prop}

\begin{proof}
Since $F(K)$ is characteristic in $K$ and $K$ is normal in $G$,
$F(K)$ is normal in $G$.
The product of two nilpotent normal subgroups of a finite group 
is again a nilpotent normal subgroup of the finite group 
(see \cite[Lemma 1.1, Chapter 6, p.217]{Gor}). 
Hence, the product $F(G) \cdot F(K)$ is 
a nilpotent normal subgroup of $G$.
By the definition of $F(G)$, we have $F(G) \cdot F(K) = F(G)$,
and the assertion $(1)$ in Proposition \ref{prop:SOV} 
indicates that $F(G) \cap F(K) =E$.
This implies that $F(K)$ is trivial.
\end{proof}

We now give some information about 
finite subgroups of $O(2) \times O(2) \times O(2)$ and $SO(3) \times SO(3)$ 
whose Fitting subgroups are cyclic,
which will be required in our proof of Lemma \ref{lem:abelian}.
It is well-known that a nontrivial finite subgroup of $SO(3)$ is isomorphic to either
\[
C_n, \;D_{2n}, \;A_4, \;S_4, \, {\rm or} \; A_5\;(n \geq 2),
\]
where $C_n$ (resp. $D_{2n}$) denotes 
a cyclic (resp. a dihedral) group of order $n$ (resp. $2n$).
Note that $A_4$ and $S_4$ contain $D_4$ 
as the unique minimal normal subgroup.

\begin{prop}\label{prop:DDD}
Let $G$ be a finite subgroup of $O(2) \times O(2) \times O(2)$.
If the Fitting subgroup $F(G)$ is cyclic then $G$ belongs to $\mathcal{G}_{1}^{2}$.
\end{prop}

\begin{proof}
We may assume that $G$ is a subgroup of
$D = D_{2l} \times D_{2m} \times D_{2n}$
for some positive integers $l$, $m$, and $n$.
It is easy to see that $D$ is an extension of a $2$-group by $F(D)$.
As $GF(D)/F(D)$ is isomorphic to $G/(G \cap F(D))$,
see \cite[Chapter 1, Theorem 2.6, p.6]{Gor},
$G$ is an extension of a $2$-group by a nilpotent group $G \cap F(D)$.
This implies that $G$ is an extension of a $2$-group by $F(G)$.
As $F(G)$ is cyclic, $G$ belongs to $\mathcal{G}_{1}^{2}$.
\end{proof}

The next result is often used to describe 
the structure of subgroups in the Cartesian product of two finite groups.

\begin{lem}[{\rm Goursat's lemma, \cite[Exercise 5, p.75]{Lang}}]\label{lem:Goursat}
Let $H$ and $K$ be finite groups and $G$ a subgroup of $H \times K$ 
such that the two projections $\pi_{1} : G \to H$ and $\pi_{2} : G \to K$ are surjective.
Let $H_{1}$ and $K_{1}$ be the kernels of $\pi_{2}$ and $\pi_{1}$, respectively.
Then the following holds:
\begin{enumerate}
 \item[$(1)$] The kernels $H_{1}$ and $K_{1}$ can be identified with 
 normal subgroups of $H$ and $K$, respectively.
 \item[$(2)$] The image of $G$ in $H/H_{1} \times K/K_{1}$ is 
 the graph of an isomorphism from $H/H_{1}$ to $K/K_{1}$.
\end{enumerate}
In particular, there exists a group exact sequence
\[
E \to H_{1} \times K_{1} \to G \to H/H_{1} (\cong K/K_{1}) \to E.
\] 
\end{lem}

\begin{prop}\label{prop:SO3SO3}
Let $G$ be a finite subgroup of $SO(3) \times SO(3)$.
If the Fitting subgroup $F(G)$ is cyclic, then we have the following:
\begin{enumerate}
\item[$(1)$] If $G$ is nonsolvable then $G$ is isomorphic to $A_5 \times L$, 
where $L$ is a finite subgroup of $SO(3)$ with cyclic $F(L)$.
\item[$(2)$] If $G$ is solvable then $G$ belongs to $\hen{G}^{2}_{1}$.
\end{enumerate}
\end{prop}

\begin{proof}
We may assume that $G$ is a subgroup of $H \times K$
having the canonical surjective projections $\pi_{1} : G \to H$ 
and $\pi_{2} : G \to K$,
where $H$ and $K$ are finite subgroups of $SO(3)$.
Let $H_{1}$ and $K_{1}$ be the normal subgroups of $H$ and $K$ 
corresponding to $\ker \pi_{2}$ and $\ker \pi_{1}$, respectively.
If $H_{1}$ or $K_{1}$ is trivial 
then $G$ is isomorphic to a finite subgroup of $SO(3)$.
Then it is easily checked that 
$G$ satisfies the assertion (1) or (2) when $F(G)$ is cyclic.
We next suppose that both $H_{1}$ and $K_{1}$ are nontrivial.
If $G$ is nonsolvable then 
it holds that $H = H_{1} \cong A_5$ or $K = K_{1} \cong A_5$.
Thus, by Lemma \ref{lem:Goursat},
$G$ is isomorphic to $A_5 \times L$ 
for a finite subgroup $L$ of $SO(3)$ with cyclic $F(L)$.
We suppose that $G$ is solvable.
By the hypothesis that $F(G)$ is cyclic,
$G$ may be a subgroup of $D_{2m} \times D_{2n}$ 
for some positive integers $m$ and $n$,
because $A_4$ and $S_4$ contain $D_4$ 
as the unique minimal normal subgroup.
As seen in the proof of Proposition \ref{prop:DDD},
$G$ is an extension of a $2$-group by $F(G)$.
As $F(G)$ is cyclic,
$G$ belongs to $\mathcal{G}_{1}^{2}$.
\end{proof}

At the end of this section,
we prepare nonabelian characteristically 
finite simple subgroups of $SO(6)$, up to isomorphism.
It is known that nonabelian finite simple subgroups of $SO(5)$ 
are isomorphic to either $A_5$ or $A_6$,
see e.g.,  \cite[Corollary1, p.742]{MecZim2} or \cite[Proposition 5]{BucKwaSch}.
Thus it remains to find nonabelian finite simple subgroups $G$ of $SO(6)$
not isomorphic to $A_5$ or $A_6$.
Such nonabelian finite simple groups $G$ have an irreducible complex representation 
of degree $3$ or $6$, which is already known:

\begin{prop}[{\rm \cite[Section 82, p.113]{Bli}, \cite[Section 3, p.773]{Lin}}]\label{prop:nonabelian}
If a nonabelian finite simple group $G$ admits 
an irreducible complex representation of degree $3$ (resp. $6$),
then $G$ is isomorphic to 
$A_5$ or $PSL(2,7)$ (resp. $A_7$, $PSL(2,7)$, $PSU(4,2)$, or $PSU(3,3)$).
\end{prop}

In Proposition \ref{prop:nonabelian},
$PSL(2,7)$ denotes the projective special linear group of order $168$,
and $PSU(3,3)$ and $PSU(4,2)$ denote the projective special unitary groups 
of orders $6048$ and $25920$, respectively.
Using the GAP system \cite{GAP},
one can check that $A_5$ and $PSL(2,7)$ have 
two irreducible complex representations of degree $3$ 
(not isomorphic to each other) and 
that each of $A_7$, $PSL(2,7)$, $PSU(4,2)$, and $PSU(3,3)$ 
has only one irreducible complex representation of degree $6$ (up to isomorphism). 
Computing the Frobenius--Schur indicators of 
the irreducible complex representations of degrees $3$ and $6$ by using \cite{GAP},
we can obtain the irreducible real representations corresponding to the complex ones.
All real representations of nonabelian finite simple groups are orientation-preserving.
The following table shows the correspondences 
between the real and complex representations
of the nonabelian finite simple groups appeared in Proposition \ref{prop:nonabelian}.

\begin{table}[!ht]
\centering
\caption{The correspondences between the real and complex representations.}
\label{table:simple}
\begin{small}
\begin{tabular}{|c|c|c|c|c|c|c|c|c|c|c|}
\hline
Complex representations & Frobenius--Schur indicators & Real representations \\ \hline
  $A_5 \to SU(3)$       &   $1$  & $A_5 \to SO(3)$ \\ \hline
  $A_5 \to SU(3)$       &   $1$  & $A_5 \to SO(3)$ \\ \hline
  $PSL(2,7) \to SU(3)$  &   $0$  & $PSL(2,7) \to SO(6)$ \\ \hline
  $PSL(2,7) \to SU(3)$  &   $0$  & $PSL(2,7) \to SO(6)$ \\ \hline
  $A_7 \to SU(6)$       &   $1$  & $A_7 \to SO(6)$ \\ \hline
  $PSL(2,7) \to SU(6)$  &   $1$  & $PSL(2,7) \to SO(6)$ \\ \hline
  $PSU(4,2) \to SU(6)$  &   $1$  & $PSU(4,2) \to SO(6)$ \\ \hline
  $PSU(3,3) \to SU(6)$    &  $-1$  & $PSU(3,3) \to SO(12)$ \\ \hline
\end{tabular}
\end{small}
\end{table}

\begin{prop}\label{prop:SO6}
A nonabelian characteristically finite simple subgroup $G$ of $SO(6)$
is isomorphic to either 
\[
A_5,\, A_6,\, A_7,\, PSL(2,7),\, PSU(4,2),\, or\, A_5 \times A_5.
\]
\end{prop}

\begin{proof}
The subgroup $G$ of $SO(6)$ is isomorphic to the Cartesian product of 
isomorphic nonabelian simple groups $Q$.
If $G \cong Q$ then we have already seen that $Q$ is isomorphic to either
\[
A_5,\, A_6,\, A_7,\, PSL(2,7),\,{\rm or}\; PSU(4,2).
\]
If $G \not\cong Q$ then $G$ is isomorphic to $A_5 \times A_5$,
because a nonabelian finite simple group contained in $SO(3)$ but not in $SO(2)$ is only $A_5$.
\end{proof}


\section{Proof of Lemma \ref{lem:nonabelian}}\label{nonabelian}

Throughout this section,
let $\varSigma$ be a homology $6$-sphere with effective $G$-action.
In this section, we study the nonempty $G$-fixed-point set $\varSigma^{G}$ 
in the case when $G$ is isomorphic to
\[
A_5,\, A_6,\, A_7,\, PSL(2,7),\, PSU(4,2),\, {\rm or}\, A_5 \times A_5,
\]
and we prove Lemma \ref{lem:nonabelian}.

The alternating group $A_5$ on the five letters $1$, $2$, $3$, $4$, $5$
has the nine conjugacy classes of subgroups and we can choose those representatives
\[
E,\; C_2, \; C_3, \; D_4, \; C_5, \; D_6, \; D_{10}, \; A_4, \, {\rm and} \; A_5
\]
so that
\[
D_4 \triangleleft A_4, \; C_2 = D_4 \cap D_6, \,{\rm and} \; \langle D_4 , D_6 \rangle = A_5,
\]
e.g., $C_2 = \langle\, (2,3)(4,5)\,\rangle$, $D_4 = \langle\, (2,3)(4,5), \, (2,4)(3,5) \,\rangle$,
$D_{6} = \langle\, (1,2,3), \, (2,3)(4,5) \,\rangle$, and
$A_{4}$ is the alternating group on the four letters $2$, $3$, $4$, $5$.
There are five irreducible $\bR[A_5]$-modules 
\[
\bR_{A_5}, \; U_{3.1}, \; U_{3.2},\; U_{4},\,{\rm and}\;U_{5},
\]
up to isomorphism, where $\bR_{A_5}$ is the trivial one and 
$U_{i}$ (resp. $U_{j,k}$) indicates $\dim U_{i}=i$ (resp. $\dim U_{j.k}=j$).
The dimensions of the $H$-fixed-point sets of the nontrivial irreducible $\bR[A_5]$-modules
for the chosen nontrivial subgroups $H$ of $A_5$ are as in Table \ref{table:A5}.

\begin{table}[!ht]
\begin{center}
\begin{small}
\caption{The dimensions of $U_{*}^{H}$ (cf. \cite[Table1.1]{Mor3}).}
\label{table:A5}
\begin{tabular}{|c||c|c|c|c|c|c|c|c|c|}
\hline
{}          & $C_2$ & $C_3$ & $C_5$ & $D_4$ & $D_6$ & $D_{10}$ & $A_4$ & $A_5$  \\ \hline
$U_{3.1}$   &  $1$  &  $1$  &  $1$  & $0$ & $0$ & $0$ & $0$  & $0$  \\ \hline
$U_{3.2}$   &  $1$  &  $1$  &  $1$  & $0$ & $0$ & $0$ & $0$  & $0$ \\ \hline
$U_{4}$     &  $2$  &  $2$  &  $0$  & $1$ & $1$ & $0$ & $1$   & $0$ \\ \hline
$U_{5}$     &  $3$  &  $1$  &  $1$  & $2$ & $1$ & $1$ & $0$  & $0$  \\ \hline
\end{tabular}
\end{small}
\end{center}
\end{table}

\begin{prop}[{\rm cf. \cite[Proposition 1.1]{KwaSch}}]\label{prop:Z2hom3A5}
Let $\Xi$ be a $\bZ_2$-homology $3$-sphere with effective $G$-action.
Then $\chi(\Xi^{G})$ is a nonnegative integer.
If the $G$-action on $\Xi$ is an odd-Euler-characteristic $G$-action,
then $G$ is isomorphic to $A_5$ and 
$\Xi^{G}$ consists of exactly one point.
\end{prop}

\begin{proof}
It suffices to prove in the case $\Xi^{G} \neq \emptyset$.
We may assume that $G$ is a finite subgroup of $O(3)$.
Let $I$ denote the unit matrix in $O(3)$.
If $G$ contains $-I$,
then $\langle -I \rangle$ is a normal subgroup of $G$ of order $2$.
By $\dim T_{x}(\Xi)^{\langle -I \rangle} =0$,
we have $\Xi^{G} \cong S^{0}$ by Lemma \ref{lem:lowsphere}.
Suppose that $-I \not\in G$ and let $\rho : G \to O(3)$ be a monomorphism.
Since the homomorphism $\det \rho \cdot \rho : G \to O(3)$ defined by
\[
\det \rho \cdot \rho (g) := \det \rho(g) \cdot \rho(g)
\]
is injective and its image is contained in $SO(3)$,
$G$ is isomorphic to a finite subgroup of $SO(3)$.
If $G$ is isomorphic to $C_{n}$, $D_{2n}$, $A_4$, or $S_4$,
where $n$ is even,
then $G$ has a nontrivial normal $2$-subgroup $P_2$.
Therefore, by Lemma \ref{lem:lowsphere},
we have $\Xi^{G} \cong S^{k}$, where $0 \leq k \leq 2$.
If $G$ is isomorphic to $C_n$ or $D_{2n}$, where $n$ is odd,
then it holds that $\dim T_{x}(\Xi)^{F(G)}=1$ 
for each $x \in \Xi^{G}$.
It follows from Lemma \ref{lem:onedim} that
$\chi(\Xi^{G})$ is a nonnegative even integer.
We finally consider the case $G \cong A_5$.
Then $T_{x}(\Xi)$ is isomorphic to 
a $3$-dimensional irreducible $\bR[A_5]$-module.
By Table \ref{table:A5} and Lemma \ref{lem:smith},
we have $\Xi^{G} \subset \Xi^{H} \cong S^0$,
where $H \cong D_4$.
Hence, it follows that $|\Xi^{G}| = 1$ or $2$.
\end{proof}

The next proposition and its corollary will help proofs of 
Proposition \ref{prop:hom6A5} and Lemma \ref{lem:A5}.

\begin{prop}[{\rm \cite[Proposition 2.4]{Tam}}]\label{prop:isomorphic}
Let $\Xi$ be a homology sphere with $G$-action.
Suppose that $G$ satisfies the following two conditions.
\begin{itemize}
  \item[(1)] Each element of $G$ is of prime-power order.
  \item[(2)] If $g \in G$ is of $2$-power order,
                 then the order of $g$ is less than or equal to $4$,
                 cf. the $8$-condition in \cite{PawSol}.
\end{itemize}
If $\Xi^{G} \neq \emptyset$ then for any $x$, $y \in \Xi^{G}$,
$T_{x}(\Xi)$ is isomorphic to $T_{y}(\Xi)$ as $\bR[G]$-modules.
\end{prop}

Note that the alternating group $A_5$ satisfies the two conditions
stated in Proposition \ref{prop:isomorphic}.

The next corollary follows immediately from Proposition \ref{prop:isomorphic}.

\begin{cor}\label{cor:pseudofree}
In the situation of Proposition \ref{prop:isomorphic},
if there exists a $G$-fixed point $x$ of $\Xi$ such that 
the tangential $G$-module $T_{x}(\Xi)$ is properly $m$-pseudofree,
then the entire $G$-action on $\Xi$ is also properly $m$-pseudofree.
\end{cor}

\begin{proof}
The assertion follows 
since all subgroups of $G$ also satisfy the two conditions
in Proposition \ref{prop:isomorphic}.
\end{proof}

\begin{prop}\label{prop:hom6A5}
If $G = A_5$ and $\varSi^{G} \neq \emptyset$, 
then $\varSigma^{G}$ satisfies at least one of the following:
\begin{itemize}
  \item[(1)] $\varSi^{G}$ is a $\bZ_2$-homology sphere of dimension $\leq 3$,
  \item[(2)] $\varSi^{G}$ is a disjoint union of two or more circles, and
  \item[(3)] $\varSi^{G}$ consists of exactly one point.
\end{itemize}
\end{prop}

The author does not know whether 
there exists an $A_5$-action on a homology $6$-sphere
such that the $A_5$-fixed-point set 
is a disjoint union of two or more circles.

\begin{proof}
By Table \ref{table:A5} and Proposition \ref{prop:isomorphic},
all the tangential $G$-modules $T_{x}(\varSi)$ 
on $G$-fixed points of $\varSigma$ 
are isomorphic to one of
\[
\bR_{A_5}^{\oplus 3} \oplus U_{3.i},
\; \bR_{A_5}^{\oplus 2} \oplus U_{4},
\; U_{3.i} \oplus U_{3.j},\,
{\rm and}\;\bR_{A_5} \oplus U_{5}\;(i, j = 1 \,{\rm or}\,2).
\]
If $T_{x}(\varSi)$ is isomorphic to 
$\bR_{A_5}^{\oplus 3} \oplus U_{3.i}$ 
(resp. $\bR_{A_5}^{\oplus 2} \oplus U_{4}$),
then $\varSi^{D_4}$ (resp. $\varSi^{C_5}$)
is a $\bZ_2$-homology $3$-sphere (resp. a sphere $S^2$) by Lemma \ref{lem:smith}.
Now, $\varSigma^{D_4}$ (resp. $\varSigma^{C_5}$) 
is a connected closed manifold and it holds that
$\dim T_{x}(\varSi)^{G} = \dim T_{x}(\varSi)^{D_4}$
(resp. $\dim T_{x}(\varSi)^{G} = \dim T_{x}(\varSi)^{C_5}$).
This implies that
$\varSigma^{G}$ coincides with $\varSigma^{D_4}$ (resp. $\varSigma^{C_5}$).
If $T_{x}(\varSi)$ is isomorphic to $U_{3.i} \oplus U_{3.j}$,
then we have $ \varSi^{G} \subset \varSi^{D_4} \cong S^0$ by Lemma \ref{lem:smith};
hence $|\varSi^{G}|=1$ or $2$.
If $T_{x}(\varSi)$ is isomorphic to $\bR_{A_5}  \oplus U_{5}$ 
(for every $x \in \varSi^{G}$),
it is easy to see that $\varSi^{G}$ is a disjoint union of some circles.
\end{proof}

\begin{prop}\label{prop:A5A5}
If $G = A_5 \times A_5$ and $\varSigma^{G} \neq \emptyset$,
then $\varSigma^{G}$ consists of exactly two points.
\end{prop}

\begin{proof}
To avoid confusion,
we will denote the first $A_5$ by $H$ and the second $A_5$ by $K$,
i.e. $G = H \times K$.
By Table \ref{table:A5},
the tangential $G$-module $T_{x}(\varSi)$ at $x \in \varSigma^{G}$ 
is isomorphic to
\[
U \otimes \bR_{K} \oplus \bR_{H} \otimes V,
\]
where $\bR_{H}$ (resp. $\bR_{K}$) is 
the $1$-dimensional $\bR[A_5]$-module
and $U$ (resp. $V$) is a $3$-dimensional irreducible $\bR[A_5]$-module.
Let $H_{1}$ (resp. $K_{1}$) denote a subgroup of $H$ (resp. $K$) isomorphic to $D_4$
and let $E_{1}$ (resp. $E_2$) denote the trivial subgroup of $H$ (resp. $K$).
Observing Table \ref{table:A5} again, we have:
\begin{itemize}
  \item[(1)] $\dim T_{x}(\varSi)^{H_1 \times E_2} = \dim T_{x}(\varSi)^{H \times E_2} = 3$,
  \item[(2)] $\dim T_{x}(\varSi)^{E_1 \times K_1} = \dim T_{x}(\varSi)^{E_1 \times K} = 3$, and
  \item[(3)] $\dim T_{x}(\varSi)^{H_{1} \times K_{1}} =0$.
\end{itemize}
By Lemma \ref{lem:smith},
we have $1 \leq |\varSi^{G}| \leq |\varSi^{H_{1} \times K_{1}}|=2$ and the closed submanifolds
$\varSi^{H \times E_2}$ and $\varSi^{E_1 \times K}$ of $\varSi$ 
are $\bZ_2$-homology $3$-spheres which intersect in exactly one or two $G$-fixed points, transversely.
Since the $\bZ_2$-intersection number of 
$\varSi^{H \times E_2}$ and $\varSi^{E_1 \times K}$ on $\varSi$ is trivial,
we obtain $|\varSi^{G}|=2$.
\end{proof}

Let $\hana{C}_{2}$ denote a subgroup of $S_5$ of order $2$ 
not contained in $A_5$.
The symmetric group $S_5$ has 
one $6$-dimensional irreducible $\bR[S_5]$-module $V_{6}$ (up to isomorphism),
and $V_{6}$ is the unique irreducible $\bR[S_5]$-module
whose restriction to $A_5$ contains $U_{3.1}$ or $U_{3.2}$.
In fact,  $V_{6}$ is induced from $U_{3.1}$,
i.e. $V_{6} \cong \bR[S_5] \otimes_{\bR[A_5]} U_{3.1}$,
and the restriction ${\rm Res}^{S_5}_{A_5}V_{6}$ is isomorphic to $U_{3.1} \oplus U_{3.2}$.
Therefore,  we have $\dim V_{6}^{\hana{C}_2}=3$.

\begin{lem}\label{lem:A5}
Suppose that $G$ contains a normal subgroup $H$ isomorphic to $A_5$ 
and $\varSigma^{G} \neq \emptyset$.
If the $G$-action on $\varSigma$ is 
an orientation-preserving effective 
odd-Euler-characteristic $G$-action, 
then it is a $2$-pseudofree one-fixed-point $A_5$-action.
\end{lem}

\begin{proof}
In this proof we identify $H$ with $A_5$.
By Proposition \ref{prop:hom6A5},
the nonempty $A_5$-fixed-point set $\varSi^{A_5}$ 
has the following three possibilities.
\begin{itemize}
 \item[$\rm(\,I\,)$] $\varSi^{A_5}$ is a $\bZ_2$-homology sphere of dimension $\leq 3$.
 \item[$\rm(I\hspace{-.01em}I)$] $\varSi^{A_5}$ is a disjoint union of two or more circles.
 \item[$\rm(I\hspace{-.15em}I\hspace{-.15em}I)$] $\varSi^{A_5}$ consists of exactly one point,
  say $\varSi^{A_5}=\{x\}$.
\end{itemize}
To complete the proof, we need to show that
in the cases $\rm(\,I\,)$ and $\rm(I\hspace{-.01em}I)$
there does not exist an odd-Euler-characteristic $G/A_5$-action 
on $\varSi^{A_5}$
and that in the case $\rm(I\hspace{-.15em}I\hspace{-.15em}I)$
the restricted one-fixed-point $A_5$-action on $\varSigma$ 
does not extend nontrivially to
an orientation-preserving one-fixed-point $G$-action on $\varSigma$.

Case $\rm(\,I\,)$.
If $\varSi^{A_5}$ is a $\bZ_2$-homology sphere of dimension $\leq 2$,
then Lemma \ref{lem:lowsphere} indicates that 
$\varSigma^{G} \cong S^{k}$, where $0 \leq k \leq 2$.
We next consider the case where $\varSigma^{A_5}$ is a $\bZ_2$-homology $3$-sphere.
Suppose that there were 
an odd-Euler-characteristic $G/A_5$-action on $\varSi^{A_5}$.
The restriction ${\rm Res}^{G}_{A_5} T_{x}(\varSi)$ is isomorphic to
$\bR_{A_5}^{\oplus 3} \oplus U_{3.i}$ ($i =1$ or $2$).
Let $\rho : G \to SO(6)$ be a tangential representation of $G$ at $x \in \varSi^{G}$ and
$\rho^{A_5} : G \to O(3)$ a subrepresentation of $\rho$ corresponding to $T_{x}(\varSi)^{A_5}$.
By Proposition \ref{prop:SOV},
$L=\ker \rho^{A_5}$ is a normal subgroup of $G$ isomorphic to 
a finite subgroup of $SO(3)$ and it contains $A_5$.
Since $A_5$ is maximal in the finite subgroups of $SO(3)$,
$L$ coincides with $A_5$.
This means that the $G/A_5$-action on $\varSigma^{A_5}$ is effective.
As $\chi((\varSigma^{A_5})^{G/A_5}) \equiv 1\;{\rm mod}\;2$,
$G/A_5$ must be isomorphic to $A_5$ by Proposition \ref{prop:Z2hom3A5},
and $G$ is thus isomorphic to $A_5 \times A_5$.
On the other hand, Proposition \ref{prop:A5A5} 
shows that if $G = A_5 \times A_5$ then $|\varSi^{G}|=2$.
This is a contradiction.

Case $\rm(I\hspace{-.01em}I)$.
As seen in the proof of Proposition \ref{prop:hom6A5},
it holds that $\dim T_{x}(\varSigma)^{A_5} =1$ for each $x \in \varSigma^{G}$.
It follows from Lemma \ref{lem:onedim} 
that $\chi(\varSigma^{G}) \equiv 0\;{\rm mod}\;2$.

Case $\rm(I\hspace{-.15em}I\hspace{-.15em}I)$.
It is easy to see that $\varSi^{G}=\{x\}$ and
${\rm Res}^{G}_{A_5} T_{x}(\varSi) \cong U_{3.i} \oplus U_{3.j}$ $(i, j = 1 \,{\rm or}\,2)$.
Since the restricted $A_5$-action on $T_{x}(\varSigma)$ is properly $2$-pseudofree,
the restricted $A_5$-action on $\varSi$ is also properly $2$-pseudofree 
by Corollary \ref{cor:pseudofree}.
Thus it would suffice to prove that 
the properly $2$-pseudofree one-fixed-point $A_5$-action on $\varSi$ 
does not extend to an orientation-preserving one-fixed-point $G$-action on $\varSi$ 
for a finite group $G$ with $[G:A_5] = p$.
As $|{\rm Out}(A_5)|=2$, $G$ is isomorphic to either $S_5$ or $A_5 \times C_p$.
If $G \cong S_5$ then $T_{x}(\varSi)$ is 
isomorphic to the irreducible $\bR[S_5]$-module $V_{6}$.
As $\dim V_{6}^{\hana{C}_{2}}=3$,
the $G$-action on $\varSi$ is orientation-reversing,
which contradicts the hypothesis that the $G$-action on $\varSi$ is orientation-preserving.
If $G \cong A_5 \times C_p$ then $T_{x}(\varSi)$ is isomorphic to the tensor product of
a $3$-dimensional irreducible $\bR[A_5]$-module $U$ 
and a $2$-dimensional $\bR[C_{p}]$-module $W$ with $W^{C_p}=\{0\}$.
As $\dim T_{x}(\varSi)^{E \times C_p}=0$,
it follows from Lemma \ref{lem:lowsphere} 
that $\varSi^{G} \cong S^0$,
which contradicts $\varSigma^{G}=\{x\}$.
Thus, we conclude that $G$ coincides with $A_5$.
\end{proof}

The proof of Case $\rm(I\hspace{-.15em}I\hspace{-.15em}I)$ 
in Lemma \ref{lem:A5} is based on the proof of 
\cite[Proposition 3.1 $\rm(\hspace{.08em}ii\hspace{.08em})$]{LaiTra}.

\vskip0.1cm

We now check the assertion in Remark \ref{rem:main}.

\begin{prop}\label{prop:main}
Let $Z$ be a group of order $2$ and $G = S_5$ or $A_5 \times Z$.
If the $G$-action on $\varSigma$ is an orientation-reversing effective one-fixed-point $G$-action
then it is properly $3$-pseudofree. 
\end{prop}
\begin{proof}

Let $\hana{C}_{2}$ be a subgroup of $S_5$ of order $2$ not contained in $A_5$,
and let $E_{1}$ and $E_2$ denote the trivial subgroups of $A_5$ and $Z$, respectively.
In this proof, we identify $A_5$ and $Z$ with $A_5 \times E_2$ and $E_1 \times Z$, respectively.
We note that ${\rm Res}^{G}_{A_5} T_{x}(\varSi)$ is isomorphic to
$U_{3.i} \oplus U_{3.j} \;(i, j = 1 \,{\rm or}\,2)$
for the unique $G$-fixed point $x$ of $\varSigma$
and that $S_5 = \langle A_5, \hana{C}_2 \rangle$ and 
$A_5 \times Z = \langle A_5, Z \rangle$.
Since the restricted one-fixed-point $A_5$-action on $\varSigma$ is $2$-pseudofree,
it suffices to prove that 
$\varSi^{\hana{C}_2}$ (resp. $\varSi^{Z}$) is a $\bZ_2$-homology $3$-sphere,
more simply,
that $\dim T_{x}(\varSi)^{\hana{C}_{2}}=3$ (resp. $\dim T_{x}(\varSi)^{Z}=3$) 
by Lemma \ref{lem:smith}.
If $G = S_5$ then $T_{x}(\varSi)$ is isomorphic to the irreducible $\bR[S_5]$-module $V_{6}$,
and we have $\dim T_{x}(\varSi)^{\hana{C}_2} = 3$.
In the case $G = A_5 \times Z$,
since the $G$-action on $T_{x}(\varSi)$ is orientation-reversing and effective,
$T_{x}(\varSi)$ must be isomorphic to
\[
U_{3.i} \otimes \bR_{Z} \oplus U_{3.j} \otimes \bR_{\pm},
\]
where $\bR_{Z}$ (resp. $\bR_{\pm}$) is the trivial (resp.  the nontrivial) 
$1$-dimensional $\bR[Z]$-module.
Thus, we have $\dim T_{x}(\varSi)^{Z} =3$.
\end{proof}

The alternating group $A_6$ on the six letters $1$, $2$, $\ldots$, $6$ has 
two $5$-dimensional irreducible $\bR[A_6]$-modules $W_{5.1}$ and $W_{5.2}$ (up to isomorphism)
and there are no faithful irreducible $\bR[A_6]$-modules of dimensions $\leq 4$.
We fix two of subgroups of $A_6$ as follows:
\[
\begin{aligned}
C_3 \times C_3 & = \langle\; (1,2,3), \;(4,5,6) \;\rangle\;{\rm and}\;\\
(C_3 \times C_3) \rtimes C_4 &= \langle\; (1,2,3), \;(4,5,6),\; (1,6)(2,4,3,5).
\end{aligned}
\]

\noindent
Note that $C_3 \times C_3$ is a normal subgroup of $(C_3 \times C_3) \rtimes C_4$.
The dimensions of the $H$-fixed-point sets of $W_{5.1}$ and $W_{5.2}$ 
for the fixed subgroups $H$ of $A_6$ are as in Table \ref{table:A6}.

\begin{table}[!ht]
\begin{center}
\begin{small}
\caption{The dimensions of $W_{5.1}^{H}$ and $W_{5.2}^{H}$.}
\label{table:A6}
\begin{tabular}{|c||c|c|c|c|c|c|c|c|c|c|c|c|c|}
\hline
{}                 &$C_3 \times C_3$ &$(C_3 \times C_3) \rtimes C_4$ \\ \hline
$W_{5.1}$          &       $1$       &            $0$             \\ \hline
$W_{5.2}$          &       $1$       &            $0$             \\ \hline
\end{tabular}
\end{small}
\end{center}
\end{table}

\begin{prop}\label{prop:A6}
If $G=A_6$ and $\varSigma^{G} \neq \emptyset$, then $\varSi^{G}$ is diffeomorphic to $S^1$.
\end{prop}

\begin{proof}
The tangential $G$-module $T_{x}(\varSi)$ at $x \in \varSigma^{G}$ is isomorphic to 
$\bR_{A_6} \oplus W_{5.1}$ or $\bR_{A_6} \oplus W_{5.2}$.
By Table \ref{table:A6}, we have
\[
\dim T_{x}(\varSi)^{C_3 \times C_3} = 2 \;{\rm and}
\; \dim T_{x}(\varSi)^{(C_3 \times C_3) \rtimes C_4} 
= \dim T_{x}(\varSi)^{A_6} = 1.
\]
It follows from Lemma \ref{lem:lowsphere} that 
$\varSigma^{(C_3 \times C_3) \rtimes C_4} \cong S^1$.
Hence, we have $\varSigma^{G} \cong S^1$.
\end{proof}

We finally consider the case that $G = PSL(2,7)$, $A_7$, or $PSU(4,2)$.
In the rest of this section, 
for simplicity, we will use the following notations:
\[
G_{1} = PSL(2,7),\; G_{2} = A_7,\, {\rm and}\;G_{3} = PSU(4,2).
\]
\noindent
Also, for a subgroup $H$ of $G$,
$H$ may be denoted by $H(G)$ to emphasize that $H$ is a subgroup of $G$.
Similarly,
an $\bR[G]$-module $V$ may be denoted by $V(G)$.
For finite groups $H$ and $K$,
if a finite group $G$ is an extension of $K$ by $H$,
then $G$ will be denoted by $H \star K$ as long as there are no confusions in discussions.

We fix the subgroups of $G_1$, $G_2$, and $G_3$ required in our proof as follows:

\[
\begin{aligned}
G_{1} &= \langle\; (3,5,7)(4,8,6),\; (1,2,6)(3,4,8) \;\rangle \cong PSL(2,7)\\
Q(G_{1}) &= \langle\; (2,3,5,4,7,8,6) \;\rangle \cong C_7,\\
P_2(G_{1})  &= \langle\; (1,2,8,4)(3,7,6,5),\;(1,3)(2,5)(4,7)(6,8) \;\rangle \cong D_8,\\
H(G_{1})    &= \langle\; (1,2,8,4)(3,7,6,5),\;(1,3)(2,5)(4,7)(6,8),\;(1,5,7)(2,4,6) \;\rangle 
\cong S_4,\;{\rm and}\\
K(G_{1})    &= \langle\; (1,2,8,4)(3,7,6,5),\;(1,3)(2,5)(4,7)(6,8),\;(2,5,4)(3,6,8) \;\rangle \cong S_4.\\
\end{aligned}
\]
\vskip0.2cm
\[
\begin{aligned}
G_2 &= \langle\; (1,2,3,4,5,6,7),\; (5,6,7) \;\rangle \cong A_7\\
Q(G_{2}) &= \langle\; (1,2,3,4,5,6,7) \;\rangle \cong C_7,\\
P_2(G_{2})  &= \langle\; (2,3)(6,7),\;(4,6)(5,7) \;\rangle \cong D_8,\\
H(G_{2})    &= \langle\; (2,3)(6,7),\;(4,6)(5,7),\;(1,3,2)(4,5)(6,7) \;\rangle \cong (D_4 \times C_3) \star C_2,\;{\rm and}\\
K(G_{2})    &= \langle\; (2,3)(6,7),\;(4,6)(5,7),\;(2,5,7)(3,4,6) \;\rangle \cong S_4.
\end{aligned}
\]
\vskip0.2cm
\[
\begin{aligned}
G_{3} &= \langle\; a_3, \;a_4, \; a_5, \;a_6, \; a_7  \;\rangle \cong PSU(4,2)\\
Q(G_{3}) &= \langle\; a_1,\;a_2 \;\rangle \cong (C_3 \times C_3 \times C_3) \star C_3,\\
P_2(G_{3}) &= \langle\; a_3,\;a_4,\;a_5 \;\rangle \cong (C_2 \times C_2 \times C_2) \star C_2,\\
H(G_{3}) &= \langle\; a_3, \;a_4, \; a_5, \; a_6 \;\rangle 
\cong ((C_2 \times C_2 \times C_2) \star C_6) \star C_2,\;{\rm and}\\
K(G_{3}) &= \langle\; a_3, \;a_4, \; a_5, \; a_7 \;\rangle \cong 
((C_2 \times C_2 \times C_2 \times C_2) \star D_4) \star C_3.
\end{aligned}
\]

Here, $a_1$, $a_2$, $a_3$, $a_4$, $a_5$, $a_6$, and $a_7$ denote the following permutations:
\[
\begin{small}
\begin{aligned}
a_1 &=(1,9,34,26,14,20,28,40,12)(2,33,16,24,11,10,30,21,38)\\
    &\hskip0.5cm(3,19,5,31,35,36,23,7,18)(4,29,25)(6,32,39)(8,17,13)(15,37,22),\\
a_2 &=(1,2,3)(5,18,36)(6,15,32)(7,21,40)(8,22,37)(9,19,33)\\
    &\hskip0.5cm(10,16,38)(11,14,35)(12,20,34)(13,17,39)(23,28,30)(24,31,26),\\
a_3 &=(1,39,10,27)(2,31,21,32)(3,20,37,13)(4,7,23,22)\\
    &\hskip0.5cm(5,19,24,29)(6,12,35,25)(9,15)(11,34,36,17)(14,16,26,40)(28,38),\\
a_4 &=(1,35)(2,14)(3,11)(4,29)(5,22)(6,10)(7,24)(8,18)(9,28)(12,39)\\
    &\hskip0.5cm(13,34)(15,38)(16,32)(17,20)(19,23)(21,26)(25,27)(30,33)(31,40)(36,37),\\
a_5 &=(1,23)(2,37)(3,21)(4,10)(5,12)(6,19)(7,39)(8,30)(9,28)(11,14)(13,32)\\
    &\hskip0.5cm(15,38)(16,17)(18,33)(20,31)(22,27)(24,25)(26,36)(29,35)(34,40),\\
a_6 &=(1,24,25)(2,33,36)(3,15,20)(4,5,10)(6,22,29)(7,35,27)(8,31,34)\\
    &\hskip0.5cm(9,26,16)(11,17,38)(13,40,18)(14,37,30)(21,28,32),\,{\rm and}\\
a_7 &=(1,22,39,37,40,25)(2,19,23,21,29,4)(3,16,12,10,7,27)\\
    &\hskip0.5cm(5,13,35,31,17,36)(6,32,34,11,24,20)(8,38)(9,30,15,33,28,18)(14,26).
\end{aligned}
\end{small}
\]

As seen in  Table \ref{table:simple},
$G_{1}$,  $G_{2}$, and $G_{3}$ have 
three, one, and one $6$-dimensional irreducible real modules,
say $V_{6.1}(G_1)$, $V_{6.2}(G_1)$, $V_{6.3}(G_1)$, $V_{6}(G_2)$, and $V_{6}(G_3)$,
respectively.
The next table shows the dimensions of the $H$-fixed-point sets 
of the $6$-dimensional irreducible real modules $V_{*}(G_{i})$ for the fixed subgroups 
$H=Q(G_{i})$, $P_2(G_{i})$, $H(G_{i})$, and $K(G_{i})$ of $G_{i}$ ($i = 1$, $2$, or $3$).

\begin{table}[!ht]
\begin{center}
\begin{small}
\caption{The dimensions of $V_{*}(G_{i})^{H}$.}
\label{table:irr6}
\begin{tabular}{|c||c|c|c|c||c|}
\hline
{} & $Q(G_{1})$ & $P_2(G_{1})$ & $H(G_{1})$ & $K(G_{1})$    \\ \hline
$V_{6.1}(G_1)$  & $0$        & $0$          & $0$        & $0$           \\ \hline 
$V_{6.2}(G_1)$  & $0$        & $0$          & $0$        & $0$           \\ \hline              
$V_{6.3}(G_1)$  & $0$        & $2$          & $1$        & $1$           \\ \hline                   
{}        & $Q(G_{2})$ & $P_2(G_{2})$ & $H(G_{2})$ & $K(G_{2})$    \\ \hline
$V_{6}(G_2)$  & $0$        & $2$          & $1$        & $1$           \\ \hline
{}    & $Q(G_{3})$ & $P_2(G_{3})$ & $H(G_{3})$ & $K(G_{3})$    \\ \hline
$V_{6}(G_3)$    & $0$        & $2$          & $1$        & $1$           \\ \hline
\end{tabular}  
\end{small}
\end{center}
\end{table}

We note that the fixed subgroups 
$Q(G_{i})$, $P_{2}(G_{i})$, $H(G_{i})$, and $K(G_{i})$ of $G_{i}$ and 
the $6$-dimensional irreducible $\bR[G_{i}]$-modules $V_{*}(G_{i})$
satisfy the following five conditions.
\begin{enumerate}
  \item[$(1)$] $Q(G_{i})$ is of prime-power order and $\dim V_{*}(G_{i})^{Q(G_{i})}=0$.
  \item[$(2)$] $H(G_{i})$ and $K(G_{i})$ generate $G_{i}$,
  i.e. $\langle H(G_{i}), K(G_{i}) \rangle = G_{i}$.
  \item[$(3)$] The intersection of $H(G_{i})$ and $K(G_{i})$ 
  coincides with the $2$-subgroup $P_2(G_{i})$.
  \item[$(4)$] $H(G_{i})$ and $K(G_{i})$ belong to $\hen{G}^{2}_{2}$ 
  (hence to $\hen{G}$).
  \item[$(5)$] $\dim V_{*}(G_{i})^{P_2(G_{i})} = 
  \dim V_{*}(G_{i})^{H(G_{i})} + \dim V_{*}(G_{i})^{K(G_{i})}$.
\end{enumerate}

Let $G$ be one of $G_{1}$, $G_{2}$, and $G_{3}$.
The following lemma shows that 
if $G$ acts effectively on $\varSigma$ with a $G$-fixed point
then the $G$-fixed-point set of $\varSigma$ consists of exactly two points.
The method of our proof of the following lemma is based on the ideas of 
\cite[Theorems 2.1 and 3.2]{BorMiz}, \cite[Lemmas 2.2 and 2.3]{Miz},
\cite[Proposition 2.4]{MorNew1}, and \cite[Lemma 3.1]{Tam}.

\begin{lem}\label{lem:S0}
Let $V$ be an $\bR[G]$-module $V$ and let
$Q$, $P_2$, $H$, and $K$ be subgroups of $G$.
Suppose that these satisfy the following five conditions:
\begin{enumerate}
\item[$(1)$] $Q$ is of prime-power order and $\dim V^{Q}=0$.
\item[$(2)$] $H$ and $K$ generate $G$.
\item[$(3)$] $P_{2}$ is of $2$-power order and is included in $H \cap K$.
\item[$(4)$] $\dim V^{P_{2}} = \dim V^{H} + \dim V^{K}$.
\item[$(5)$] If $\dim V^{H} = 0$ (resp. $\dim V^{K} = 0$) 
then $K$ (resp. $H$) belongs to $\hen{G}$.
\end{enumerate}

\noindent
If $G$ acts on a homology sphere $\Xi$ and there exists 
a $G$-fixed point $x$ of $\Xi$
such that $T_{x}(\Xi)$ is isomorphic to $V$ as $\bR[G]$-modules,
then $\Xi^{G}$ consists of exactly two points, i.e. $\Xi^{G} \cong S^0$. 
\end{lem}

\begin{proof}
By Lemma \ref{lem:smith},
we have $\{x\} \subset \Xi^{G} \subset \Xi^{Q} \cong S^0$.
Thus we need to check that $\Xi^{G} \neq \{x\}$ 
in the following four cases.
\begin{enumerate}
  \item[$\rm(\,I\,)$] $\dim T_{x}(\Xi)^{P_{2}} 
  = \dim T_{x}(\Xi)^{H} 
  = \dim T_{x}(\Xi)^{K} =0$.
  \item[$\rm(I\hspace{-.01em}I)$] 
  $\dim T_{x}(\Xi)^{P_{2}} 
  = \dim T_{x}(\Xi)^{H} >0$ 
  and $\dim T_{x}(\Xi)^{K}=0$.
  \item[$\rm(I\hspace{-.15em}I\hspace{-.15em}I)$] 
  $\dim T_{x}(\Xi)^{P_{2}} 
  = \dim T_{x}(\Xi)^{K} >0$ 
  and $\dim T_{x}(\Xi)^{H}=0$.
  \item[$\rm({I}\hspace{-1.2pt}\mathrm{V})$] 
  $\dim T_{x}(\Xi)^{H}>0$ and $\dim T_{x}(\Xi)^{K}>0$.
\end{enumerate}

We note that it follows from the hypotheses (2) and (3)
that $\Xi^{P_2} \supset \Xi^{H} \cap \Xi^{K} = \Xi^{G}$.

Case $\rm(\,I\,)$.
By the hypothesis that both of $H$ and $K$ belong to $\hen{G}$, 
it follows from Lemma \ref{lem:euler} that $\chi(\Xi^{H}) \neq 1$ and $\chi(\Xi^{K}) \neq 1$.
By Lemma \ref{lem:smith},
$\Xi^{P_{2}}$ consists of exactly two points 
(hence $|\Xi^{H}| = |\Xi^{K}| = 2$).
Thus we have
\[
\Xi^{P_{2}} =  \Xi^{H} \cap \Xi^{K} = \Xi^{G} \neq \{x\}.
\]

Case $\rm(I\hspace{-.01em}I)$.
Since $\Xi^{P_{2}}$ is a connected closed manifold 
(by Lemma \ref{lem:smith}),
$\Xi^{P_{2}}$ coincides with $\Xi^{H}$.
Thus it holds that
\[
\Xi^{G} = \Xi^{H} \cap \Xi^{K} = \Xi^{P_{2}} \cap \Xi^{K} = \Xi^{K}.
\]
As $K \in \hen{G}$, it follows from Lemma \ref{lem:euler} 
that $\Xi^{G} = \Xi^{K} \neq \{x\}$.

Case $\rm(I\hspace{-.15em}I\hspace{-.15em}I)$.
We can prove that $\Xi^{G} = \Xi^{H} \neq \{x\}$
in the same way as Case $\rm(I\hspace{-.01em}I)$
by replacing $H$ with $K$.

Case $\rm({I}\hspace{-1.2pt}\mathrm{V})$.
Suppose that $\Xi^{G}$ were a singleton, i.e.  $\Xi^{G} = \{x\}$.
Let $\Xi_{x}^{H}$ and $\Xi_{x}^{K}$ denote the connected components 
of $\Xi^{H}$ and $\Xi^{K}$ including $x$, respectively.
The closed submanifolds $\Xi_{x}^{H}$ and $\Xi_{x}^{K}$ 
of a $\bZ_2$-homology sphere $\Xi^{P_{2}}$ intersect
at the unique $G$-fixed point $x$, transversely.
On the other hand,
the $\bZ_2$-intersection number 
of $\Xi_{x}^{H}$ and $\Xi_{x}^{K}$ on $\Xi^{P_2}$ must be trivial,
which contradicts the assumption that 
$\Xi^{G} = \{x\}$. Thus we obtain $\Xi^{G} \neq \{x\}$.
\end{proof}

\begin{prop}\label{prop:A7PSL27PSU42}
If $G = G_{1}$, $G_{2}$, or $G_{3}$ and $\varSigma^{G} \neq \emptyset$,
then $\varSi^{G}$ is diffeomorphic to $S^0$.
\end{prop}

\begin{proof}
Since $G$ has no faithful irreducible $\bR[G]$-modules of dimensions $\leq 5$,
the tangential $G$-module $T_{x}(\varSigma)$ is isomorphic to 
an irreducible $\bR[G]$-module of dimension $6$.
Therefore, Proposition \ref{prop:A7PSL27PSU42} follows immediately from Lemma \ref{lem:S0}.
\end{proof}

\begin{com}
As mentioned in the proof of \cite[Proposition 3.1]{LaiTra}, 
the linear actions on 
$V_{6.1}(G_1)$, $V_{6.2}(G_1)$, $V_{6.3}(G_1)$, $V_{6}(G_2)$, and $V_{6}(G_3)$
are not $2$-pseudofree, but properly $4$-pseudofree.
Therefore, the assertion of Proposition \ref{prop:A7PSL27PSU42} 
is not included in the result in \cite{MorNew2}.
\end{com}

\noindent
{\bf Proof of Lemma \ref{lem:nonabelian}.}
Let $H$ be a nonabelian minimal normal subgroup of $G$ and 
let $\varSigma$ be a homology $6$-sphere 
with orientation-preserving effective odd-Euler-characteristic $G$-action.
As $\varSigma^{G} \neq \emptyset$, 
we may assume that $G$ is a finite subgroup of $SO(6)$.
It follows from Proposition \ref{prop:SO6} that $H$ is isomorphic to either
\[
A_5,\, A_6,\, A_7,\, PSL(2,7),\, PSU(4,2),\, {\rm or}\; A_5 \times A_5.
\]
If $H$ is not isomorphic to $A_5$
then Propositions \ref{prop:A5A5}, \ref{prop:A6}, and \ref{prop:A7PSL27PSU42} show
that $\varSigma^{H}$ is a sphere of dimension $\leq 1$.
Therefore, by Lemma \ref{lem:lowsphere},
$\varSigma^{G}$ is also a sphere of dimension $\leq 1$.
If $H$ is isomorphic to $A_5$ then 
the assertion of Lemma \ref{lem:nonabelian}
is itself the assertion of Lemma \ref{lem:A5}.
\qed

\section{Proof of Lemma \ref{lem:abelian}}\label{abelian}

Throughout this section, let $\varSigma$ be a homology $6$-sphere
and it has an orientation-preserving effective $G$-action.

For a prime $p$, let $O_{p}(G)$ denote the $p$-Sylow subgroup of $F(G)$,
which is often called the {\it $p$-core} of $G$.
The $p$-core $O_{p}(G)$ of $G$ is the largest normal $p$-subgroup of $G$
and $F(G)$ is the Cartesian product of the nontrivial $p$-cores $O_{p}(G)$ of $G$.

\vskip0.1cm

We give a proof of Lemma \ref{lem:abelian} by using the results proved so far.

\vskip0.1cm

\noindent
{\bf Proof of Lemma \ref{lem:abelian}.}
If $\varSigma^{G} = \emptyset$ then the assertion is obvious.
Suppose that $\varSigma^{G} \neq \emptyset$ and
then we may assume that $G$ is a finite subgroup of $SO(6)$.
Furthermore, by Lemma \ref{lem:nonabelian}, 
we can assume that $G$ contains no nonabelian minimal normal subgroups.

Let $\rho : G \to SO(6)$ be a tangential representation of $G$ at $x \in \varSi^{G}$ and 
$\rho_{F(G)}: G \to O(n)$ a subrepresentation of $\rho$ 
corresponding to $T_{x}(\varSigma)_{F(G)}$.
By Propositions \ref{prop:SOV} and \ref{prop:FG}, we have:
\begin{itemize}
 \item[(1)] $K = \ker \rho_{F(G)}$ is isomorphic to 
 a finite subgroup of $SO(m)$ with $F(K)=E$, and
 \item[(2)] $G/K$ is isomorphic to a finite subgroup of $O(n)$.
\end{itemize}
Here, $m =\dim T_{x}(\varSigma)^{F(G)}$ and $n = \dim T_{x}(\varSigma)_{F(G)}$.
The assumption that $G$ contains no nonabelian minimal normal subgroups
implies that $K$ also contains no nonabelian minimal normal subgroups.
As $F(K)$ is trivial, $K$ must be trivial.
Thus, $G$ is isomorphic to a finite subgroup of $O(n)$.

Lemma \ref{lem:abelian} will be completed by proving the two assertions below.
\vskip0.1cm

\noindent
{\bf Assertion 1.} If $F(G)$ is noncyclic,
then $\chi(\varSi^{G})$ is an even integer.

\vskip0.1cm

As a $p$-core $O_{p}(G)$ is noncyclic,
we have either 
\[
\dim T_{x}(\varSi)^{O_{p}(G)} \leq 2 \;\;{\rm or}\; \dim T_{x}(\varSi)^{O_{2}(G)} = 3.
\]
If $\dim T_{x}(\varSi)^{O_{p}(G)} \leq 2$
then it follows from Lemma \ref{lem:lowsphere} that $\varSi^{G} \cong S^{k}$,
where $0 \leq k \leq 2$.
We next consider the case $\dim T_{x}(\varSi)^{O_{2}(G)} =3$.
Since $F(G)/O_{2}(G)$ is of odd order,
it holds that $\dim T_{x}(\varSigma)^{F(G)} = 1$ or $3$ for each $x \in \varSigma^{G}$. 
If $\dim T_{x}(\varSigma)^{F(G)} = 1$ for each $x \in \varSigma^{G}$,
then Lemma \ref{lem:onedim} indicates that 
$\chi(\varSigma^{G})$ is a nonnegative even integer.
If $\dim T_{x}(\varSigma)^{O_{2}(G)} = \dim T_{x}(\varSigma)^{F(G)} = 3$,
then $\varSigma^{F(G)}$ is a $\bZ_2$-homology $3$-sphere with $G/F(G)$-action. 
Now, $G$ is isomorphic to a finite subgroup of $O(3)$
containing no nonabelian minimal normal subgroups.
In fact, such a finite group $G$ is solvable (thus $G/F(G)$ is so).
Thus, by Proposition \ref{prop:Z2hom3A5},
we see that $\chi((\varSigma^{F(G)})^{G/F(G)})$ is a nonnegative even integer.

\vskip0.1cm

\noindent
{\bf Assertion 2.} If $F(G)$ is nontrivial and cyclic,
then $\chi(\varSi^{G})$ is an even integer.
\vskip0.1cm

Since $F(G)$ is cyclic and 
the restricted $F(G)$-action on $T_{x}(\varSigma)$ is orientation-preserving and effective,
we have $\dim T_{x}(\varSigma)^{F(G)} = 0$, $2$, or $4$.
To complete the proof,
we need to show that $\chi(\varSi^{G}) \equiv 0\;{\rm mod}\;2$ in the following three cases.

\begin{itemize}
  \item[$\rm(\,I\,)$] $\dim T_{x}(\varSi)^{F(G)}=4$ for some $x \in \varSi^{G}$.
  \item[$\rm(I\hspace{-.01em}I)$] $\dim T_{x}(\varSi)^{F(G)}=2$ for some $x \in \varSi^{G}$.
  \item[$\rm(I\hspace{-.15em}I\hspace{-.15em}I)$] $\dim T_{x}(\varSi)^{F(G)}=0$ 
  for some $x \in \varSi^{G}$.
\end{itemize}

Case $\rm(\,I\,)$.
Since $G$ is isomorphic to a finite subgroup of $O(2)$,
$G$ belongs to $\mathcal{G}_{1}^{2}$.
Therefore we have $\chi(\varSigma^{G}) \equiv 0\;{\rm mod}\;2$ by Lemma \ref{lem:euler}.

Case $\rm(I\hspace{-.01em}I)$.
We will identify $G$ with a finite subgroup of $O(4)$.
Let $L$ be a finite subgroup of $G$ 
contained in $SO(4)$ with $[G:L] \leq 2$ and
let $Z$ be the center of $SO(4)$ 
generated by the central inversion $- I$ in $SO(4)$,
i.e. $Z = \langle -I \rangle$.
The subgroup $L$ of $SO(4)$ acts effectively 
on $T_{x}(\varSigma)$ 
via $\rho : G \to SO(6)$.
If $-I \in L$ then $Z$ is characteristic 
in $L$ (thus $Z$ is normal in $G$) 
and we have $\dim T_{x}(\varSigma)^{Z} \leq 2$.
Therefore, by Lemma \ref{lem:lowsphere}, 
one has $\varSigma^{G} = (\varSigma^{Z})^{G/Z} \cong S^k$,
where $0 \leq k \leq 2$.
If $-I \not\in L$ then 
the double covering homomorphism $\pi : SO(4) \to SO(3) \times SO(3)$ 
maps $G$ isomorphically onto a subgroup of $SO(3) \times SO(3)$.
Thus $L$ is isomorphic to a finite subgroup of $SO(3) \times SO(3)$. 
As the subgroup $F(L)$ of $F(G)$ is cyclic,
$L$ belongs to $\hen{G}^{2}_{1}$ 
by Proposition \ref{prop:SO3SO3},
because $L$ contains no nonabelian minimal normal subgroups.
Therefore, by Lemmas \ref{lem:modp} and  \ref{lem:euler},
we have $\chi(\varSigma^{G}) \equiv 
\chi(\varSigma^{L}) \equiv 0\;{\rm mod}\;2$.

Case $\rm(I\hspace{-.15em}I\hspace{-.15em}I)$.
As $F(G)$ is cyclic, every $p$-core $O_{p}(G)$ is cyclic.
Therefore, for every prime $p$ dividing the order of $F(G)$,
we have $\dim T_{x}(\varSigma)^{O_{p}(G)}=0$, $2$, or $4$.
If $\dim T_{x}(\varSigma)^{O_{p}(G)}=0$ or $2$
then it follows from Lemma \ref{lem:lowsphere} 
that $\varSigma^{G} \cong S^{k}$, where $0 \leq k \leq 2$.
We next consider the case that
$\dim T_{x}(\varSigma)^{O_{p}(G)} = 4$ 
for every prime $p$ dividing the order of $F(G)$.
Let $\rho^{O_{p}(G)} : G \to O(4)$ and $\rho_{O_{p}(G)} :G \to O(2)$ be subrepresentations of $\rho$
corresponding to $V^{O_{p}(G)}$ and $V_{O_{p}(G)}$,
respectively, 
and let $H_{p}= \ker \rho^{O_{p}(G)}$ and $K_{p} = \ker \rho_{O_{p}(G)}$.
By $\dim T_{x}(\varSigma)^{F(G)}=0$, 
there is a prime $q$ such that $O_{q}(G) \cap H_{p} = E$ 
and that $\dim T_{x}(\varSigma)^{O_{p}(G) \times O_{q}(G)}=2$. 
Then, 
$\rho^{O_{p}(G) \times O_{q}(G)}$ is a subrepresentation of $\rho^{O_{p}(G)}$
and its complement  is $\rho_{O_{q}(G)}$, i.e. 
\[
\rho^{O_{p}(G)} = \rho_{O_{q}(G)} \oplus \rho^{O_{p}(G) \times O_{q}(G)} :
G \to O(2) \times O(2) \to O(4).
\]
Let $K_{q} = \ker \rho_{O_{q}(G)}$ and $H_{pq} = \ker \rho^{O_{p}(G) \times O_{q}(G)}$.
Then, $H_{p}$ coincides with $ K_{q} \cap H_{pq}$.
Since $H_{p} \cap K_{p} = K_{q} \cap H_{pq} \cap K_{p} = E$, we have a monomorphism
\[
G \to G/K_{q} \times G/H_{pq} \times G/K_{p} \to O(2) \times O(2) \times O(2).
\]
Thus, $G$ is isomorphic to a finite subgroup of $O(2) \times O(2) \times O(2)$.
As $F(G)$ is cyclic,
it follows from Proposition \ref{prop:DDD} and Lemma \ref{lem:euler} that $\chi(\varSigma^{G}) \equiv 0\;{\rm mod}\;2$.
\qed

\end{document}